\documentclass[12pt,reqno]{amsart}
\usepackage[utf8]{inputenc}
\usepackage{amsmath}
\usepackage{amscd}
\usepackage{graphics}
\usepackage{latexsym}
\usepackage{color}
\usepackage{verbatim}
\usepackage{extarrows}
\usepackage{latexsym}
\usepackage{enumerate}
\usepackage{arydshln}
\usepackage{soul}
\usepackage{tikz}
\usepackage{tikz-cd}
\usepackage{ulem} 
\usepackage{MnSymbol}

\usepackage{dsfont}
\usepackage{marvosym}
\usepackage{mathrsfs}

\definecolor{rose}{rgb}{0.82, 0.1, 0.26}
 	\definecolor{blu}{rgb}{0.36, 0.54, 0.66}
 	\definecolor{mor}{rgb}{0.55, 0.0, 0.55}

   \usepackage{hyperref}
 \hypersetup{
    colorlinks=true,
    linkcolor=rose, 
    filecolor=blue     }

\usepackage[alphabetic]{amsrefs}

    \hypersetup{ citecolor=mor} 
  \hypersetup{ filecolor=blue}

\usepackage[margin=1in]{geometry}
\usepackage[all]{xy}
\theoremstyle{plain}
\newtheorem{theorem}{Theorem}[section]
\newtheorem{thm}[theorem]{Theorem}
\newtheorem{cor}[theorem]{Corollary}
\newtheorem{lem}[theorem]{Lemma}
\newtheorem{prop}[theorem]{Proposition}

\newtheorem{defi}[theorem]{Definition}

\theoremstyle{definition}

\newtheorem{rmk}[theorem]{Remark}

\newtheorem*{thm1}{Theorem A}
\newtheorem*{thm2}{Theorem B}
\newtheorem*{coro}{Corollary}

\theoremstyle{remark}

\newcommand{\PP}{\mathds{P}}
\newcommand{\PPP}{\PP^{3}}
\newcommand{\cc}{\mathscr{C}}
\newcommand{\ccc}{\mathscr{C}^{h}}
\newcommand{\hh}{\mathscr{H}}
\newcommand{\aaa}{\mathscr{A}}
\newcommand{\rr}{\mathscr{R}}
\newcommand{\oc}{\overline{\mathscr{C}_{3}}}
\newcommand{\ii}{\mathscr{I}}

\usepackage{graphicx}
\usepackage{subcaption}
\usepackage{wrapfig}

\title{Liaison theory and the birational geometry of the Hilbert scheme of curves in $\PPP$}
\author{ Montserrat Vite}
\address{Instituto de Matem\'aticas, Universidad Nacional Aut\'onoma de M\'exico, Oaxaca de Ju\'arez, M\'exico}
\email{montserrat@matmor.unam.mx} 
\begin{document}

\begin{abstract}
  In the Hilbert scheme of curves of degree $d_{r}=\frac{r(r+1)}{2}$ and arithmetic genus $g_{r}=\frac{r(r+1)(2r-5)}{6}+1$ in $\PPP$ we prove that there exists a unique component of arithmetically Cohen-Macaulay curves denoted by $\overline{\cc_{r}}$. For $r\geq 3$, we verify that the subvariety of curves in $\overline{\cc_{r}}$ with Rao module of rank one always contains a reducible divisor. In particular, in the case of curves of degree $6$ and genus $3$ we prove that this subvariety is a reducible divisor. Furthermore, the components of such divisor are linearly independent and each component generates an extremal ray of the effective cone $\overline{\text{Eff}(\oc)}$. 
\end{abstract}

\maketitle

 \tableofcontents
\section*{Introduction}

Two curves in $\PPP$ are \textit{geometrically directly linked} if their union is the complete intersection of two surfaces. This definition can be extended to define an equivalence relation, \textit{linkage equivalence} (for more details see  \cite{mig2}). In the 1940’s Apery and Gaeta showed that a smooth curve $C$ in $\PPP$ is in the linkage class of a complete intersection if and only if it is \textit{arithmetically Cohen-Macaulay (ACM for short)} (cf. \cite{apery} and \cite{gaeta}). This result was extended to arbitrary codimension two subschemes of projective space by Peskine and Szpiro \cite{pesk}, putting the theory of liaison into the framework of scheme theory. We are interested in studying families related by links, in particular, we ask ourselves:
\begin{center}
    \textit{How do the linkage classes behave in (flat) families?}
\end{center}

We study this question, and exhibit families of curves in the linkage class of two disjoint lines as limits of ACM curves. We can phrase our first result as follows:

\begin{thm1}[Corollary \ref{corP}] 
Let $\overline{\cc_{r}}$ be the closure of the family of curves in the linkage class of a line in the Hilbert scheme of curves of degree $d_{r}=\frac{r(r+1)}{2}$ and genus $g_{r}=\frac{r(r+1)(2r-5)}{6}+1$, then $\overline{\cc_{r}}$ is irreducible and has two irreducible divisors, each parametrizes curves in the linkage class of two disjoint lines.
\end{thm1}

This is not obvious since the Hilbert scheme is reducible, which constitutes a major obstacle this paper had to overcome. Furthermore, it is the first result that uses Liaison Theory in describing effective divisors of the Hilbert scheme of curves in projective 3-space, as far as we know. This Theorem is an effective instance of a result by Ballico and Bolondi (cf. \cite{babo}*{Corollary 2.3}).

There is a further reason that attracted our attention to our previous theorem: its implications in birational geometry.

Indeed, consider the following two families of 1-dimensional sub schemes of degree 6 and arithmetic genus 3: the family of hyperelliptic curves of degree $6$, denoted by $\ccc$, and the family of curves that parameterize the union of a plane quartic with two incident skew lines, denoted by $\aaa$. The closure of these two families are the divisors of the theorem A on $\oc$. The component $\oc$ may not be normal, thus, we have to take the normalzation of $\oc$, denoted by $\mathscr{B}$, that contains all the locally Cohen-Macaulay curves, in particular, the families $\overline{\ccc}$ and $\overline{\aaa}$. In this case, if $N^{1}(\mathscr{B})$ is the space of divisors of $\mathscr{B}$ modulo numerical equivalence, then we prove:

\begin{thm2}[Theorem \ref{li}] \label{intro1}
The classes of $\overline{\aaa}$ and $\overline{\ccc}$ in $N^{1}(\mathscr{B})$ are linearly independent and each of them spans an extremal ray of the cone of effective divisors: \[ \overline{\text{Eff}(\mathscr{B})}\subseteq N^1(\mathscr{B}).\]
\end{thm2}

As a corollary, we compute the dimension of the space $N^{1}(\mathscr{B})$:
\begin{coro}[Corollary \ref{T1}]
    The dimension of the vector space $N^{1}(\mathscr{B})$ is $3$.
\end{coro}

In general, these divisors allow us to give a lower bound for the dimension of the vector space of divisors of $\overline{\cc_{r}}$ modulo numerical equivalence:
\begin{coro}
    For every $r\geq 4$ we have that
    $$dim\,N^{1}(\overline{\cc_{r}})\geq 2.$$
\end{coro}

Theorem A is proved using basic notions of liaison theory and the theory developed in \cite{mdp2}. On the other hand, to prove Theorem B we construct a birational map that contracts both families.

\textbf{Structure of the paper.} The first section presents the preliminaries of Liaison Theory, the definitions of the objects that we study throughout and verify that there exists a unique component of ACM curves in $\hh_{r}$. Section 2 is devoted to the case of curves of degree $6$ and genus $3$ and focuses on the main component. In  this section, we prove  Theorem B. The last two sections are dedicated to the study of the general case; in particular, in section 4 we prove Theorem A.

\textbf{Acknowledgements.} I am grateful to my advisor C\'esar Lozano Huerta for supporting and guiding me through all these years.  Thanks also to Manuel Leal for all
the useful conversations and more. I would like to thank Dra. Leticia Brambila and Dr. Abel Castorena for their useful comments on the first version of this paper. Special thanks to Robin Hartshorne whose talk at the Seminario Nacional de Geometría Algebraica helped me shape this work. I am also grateful to the Pontificia Universidad Católica de Chile, and especially to Mauricio Bustamante, for their hospitality. Finally, thanks to CONAHCYT for financial support, and IMUNAM for supporting me throughout my PhD. \\

\textbf{Notation.} Throughout this paper, $k$ is an algebraically closed  field of characteristic zero. All varieties and subschemes will be assumed to be projective. We shall
denote by $S$ the homogeneous polynomial ring $k[x,y,z,w]$ and we let $\PP^{3}:=\PP^{3}_{k}=Proj(S)$  stand for the projective $3$-space. In this work, by \textit{curve} we mean a one-dimensional closed subscheme of $\PPP$ that is locally Cohen-Macaulay (lcm). These are closed subschemes of dimension one that may be reducible and non-reduced but that have no isolated or embedded points.

\section{Preliminaries}
In this Section, we review the main results from Liaison Theory that we use and establish the notation used throughout. \\

Let us start with the most simple curve in $\PP^{3}$, the line. Given two different quadrics in $\PPP$ that contain a fixed line, their intersection is the union of the line with a curve of degree $3$ and genus $0$, a twisted cubic. If we take two different cubic surfaces that contain this twisted cubic and intersect them, one obtains a residual curve of degree $6$ and genus $3$. We can continue iterating this process to produce an infinite family of curves and all these curves are linked to the line in the sense of the next definition.

\begin{defi}  Two curves   $C$ and $C^{\prime}$ in $\PP^{3}$ are \textit{directly geometrically linked } (or simply directly linked) by the complete intersection of two surfaces $X$ and $X^{\prime}$ if $C\cup C^{\prime}=X\cap X^{\prime}$. The curves $C, C^{\prime}$ are \textit{linked} if there exists a  finite number of curves $C_{1}, \ldots ,C_{m}$ in $\PP^{3}$ such that $C_{i}$ is directly linked to $C_{i+1}$ for all $i$, with $C=C_{1}$ and $C^{\prime}=C_{m}$. 
\end{defi}

The previous definition induces an equivalence relation called \textit{linkage equivalence}. The linkage classes can be identified via the following $S$-module of finite length:

\begin{defi}
Let $C$ be a curve in $ \PP^{3}$. The \textit{Rao module} of $C$ is defined by:
$$M(C):=\bigoplus_{n\in{\mathds{Z}}}H^{1}(\PP^{3},\mathscr{I}_{C}(n)).$$
\end{defi}

 We are interested in studying families of curves linked to a fixed family $\mathscr{C}$. For this reason, we set the notation to refer to these types of families.

 \begin{defi}  Given a family of curves with constant cohomology $\cc$ in a Hilbert scheme of curves in $\PP^{3}$ we denote by $\mathscr{L}_{s}\cc$ the family of all curves linked to elements of $\cc$ by the complete intersection of two surfaces of degree $s$, we say that the family $\mathscr{L}_{s}\cc$ is a \textit{linked family}. In many cases, we iterate this construction and write $\mathscr{L}_{s}^{n+1}\cc$ instead of $\mathscr{L}_{s}(\mathscr{L}_{s-1}^{n}\cc)$. 
\end{defi}

\begin{rmk}
In terms of \cite{mdp2}*{VII}, $\cc$ is an element of $H_{\gamma ,M}(Y)$ for some closed scheme $Y$ and $\mathscr{L}\cc$ is an element of $H_{\gamma^{\prime},M^{\prime}}(Y)$ such that the schemes of flags are isomorphic, $\mathcal{D}_{\gamma ,M,s,s}\cong \mathcal{D}_{\gamma^{\prime},M^{\prime},s,s}$.
\end{rmk}

Observe that, if a curve $C$ in the family $\cc$ has degree $d$ and genus $g$, then the curves $C^{\prime}$ in the family $\mathscr{L}_{s}\cc$ linked to $C$ have degree $d^{\prime}=s^{2}-d$ and genus $g^{\prime}=\frac{(d^{\prime}-d)(2s-4)}{2}+g$ \cite{mdp2}*{III,Prop 1.2}. Furthermore, we have that $H^{1}(\PP^{3},\mathscr{I}_{C^{\prime}}(n))$ is isomorphic to the dual of $H^{1}(\PP^{3},\mathscr{I}_{C}(2s-4-n))$ for each $n$. In particular, this implies that $M(C^{\prime})=0$ if $M(C)=0$.\\

Returning to the curves linked to the line, by the previous observation, inductively we have that for every $r$, we obtain a curve of degree $d_{r}=\frac{r(r+1)}{2}$ and genus $g_{r}=\frac{r(r+1)(2r-5)}{6}+1$. We are interested in studying the Hilbert scheme of curves with each of these invariants. Thus, for any positive integer $r$, we set the numbers:
\begin{align*}
 d_{r}=\frac{r(r+1)}{2} \quad \text{  and  } \quad g_{r}=\frac{r(r+1)(2r-5)}{6}+1. 
 \end{align*} 
Let us consider the Hilbert scheme $Hilb_{p_{r}(t),3}$ of one dimensional closed subschemes of $\PPP$ with Hilbert polynomial $p_{r}(t)=d_{r}t+(1-g_{r})$. The set of locally Cohen-Macaulay curves of degree $d_{r}$ and arithmetic genus $g_{r}$ is denoted by $\mathscr{H}_{r}^{lcm}\subseteq Hilb_{p_{r}(t),3} $. We refer to the curves in $\mathscr{H}_{r}^{lcm}$  as \textit{triangular curves}.

 Triangular curves satisfy: 
\begin{enumerate}
\item Their degrees are triangular numbers,
\item We have that $\mathscr{L}_{r+1} \mathscr{H}_{r}^{lcm} \subseteq \hh^{lcm}_{r+1}$ for all $r$.
\end{enumerate}  

\begin{defi}
Let $\mathscr{H}_{r}$ be the closure of triangular curves $\mathscr{H}_{r}^{lcm}$ in $Hilb_{p_{r}(t),3}$.
\end{defi} 

The first two cases are very well known:
\begin{enumerate}
\item The Hilbert scheme $\hh_{1}$ is the Grassmannian of lines in $\PP^{3}$, it is smooth and irreducible of dimension $4$.
\item The scheme $\hh_{2}$ is the component of the Hilbert scheme of curves of degree $3$ and genus $0$ whose generic element is a twisted cubic. This scheme is smooth and irreducible of dimension $12$ (cf. \cite{dc}, \cite{ct} and \cite{ps}).
\end{enumerate}

Our purpose is to study the cases for $r\geq 3$, in particular, we focus on a special component in each of these spaces. To describe this component let's consider the family in $\hh_{r}$ that parametrizes curves $C$ with minimal free resolution: 
\begin{equation} \label{resacm}
\xymatrix{ 0\ar[r]&\mathcal{O}_{\PP^{3}}(-(r+1))^{r}\ar[r]^{M_{r}}& \mathcal{O}_{\PP^{3}}(-r)^{r+1}\ar[r] &  \mathscr{I}_{C} \ar[r] &0 }\end{equation} 
where $\mathscr{I}_{C}$ is the ideal sheaf of $C$. It follows from this resolution that the degree of an element in this family is $d_{r}$ and the genus is $g_{r}$. 

\begin{defi}
We denote by $\cc_{r}$ the family of curves in $\hh_{r}$ with minimal free resolution as  (\ref{resacm}). 
\end{defi} 
For curves $C\in{\cc_{r}}$, the Rao module $M(C)$ is trivial, i.e., $M(C)\equiv 0$, which means that the curves in $\cc_{r}$ are ACM curves. We also know that these curves are smooth points in $\hh_{r}$ by  \cite{ell}*{Thm 2}. We claim that these families form the unique component of ACM curves in $\hh_{r}$ for each $r>0$. Before seeing that, we need a technical lemma:

\begin{lem} \label{lemidealACM}
    Let $C\in{\hh_{r}}$ an ACM curve, then:
    \begin{enumerate}
        \item We have that: $h^{0}(\PPP,\ii_{C}(r))=r+1,$
        \item the minimal degree of a surface that contains $C$ is $r$, that means:
        $$s_{C}=min\{n|h^{0}(\PPP,\ii_{C}(n))\not= 0\}=r$$
    \end{enumerate}
\end{lem}
\begin{proof}
    Since $C$ is an ACM curve, it is not special, therefore:
    \begin{align*}
        h^{0}(\PPP,\ii_{C}(r))&=h^{0}(\PPP,\mathcal{O}_{\PPP}(r))-h^{0}(C,\mathcal{O}_{C}(r))\\
        &=\binom{r+3}{3}-(1-g_{r}+rd_{r})=r+1
    \end{align*}
    That proves the first point. On the other hand,
    \begin{align*}
        h^{0}(\PPP,\ii_{C}(r-1))&=h^{0}(\PPP,\mathcal{O}_{\PPP}(r-1))-h^{0}(C,\mathcal{O}_{C}(r-1))\\
        &=\binom{r+2}{3}-(1-g_{r}+(r-1)d_{r})=0
    \end{align*}
    thus we have that $r-1<s_{C}\leq r$.\\
\end{proof}

\begin{thm} \label{acm}
The family $\overline{\cc_{r}}$ is the only irreducible component in $\hh_{r}$ of ACM curves.
\end{thm}
\begin{proof}
By \cite{mig2} the existence of ACM components are in correspondence with the existence of $h$-vectors of curves $C$ that satisfies:
\begin{enumerate}
    \item $h(0)=1$,
\item There is an integer $s\geq 1$, equal to the least degree of a surface containing $C$ such that
$$\left\{
  \begin{array}{ll}
		 h(n)=n+1    & \mathrm{for} \; 0\leq n\leq s-1; \\
	    h(n)\geq h(n+1)  & \mathrm{for} \; n \geq s-1;\\ 
         h(n)=0    & \mathrm{for} \; n \gg 0. 
		 \end{array}
	     \right.$$
      \item $h^{0}(\PPP,\ii_{C}(s))=s+1-h(s)$.
\end{enumerate}

In the case of the curves in $\cc_{r}$, the $h$-vector is $ \{(1, 2, \ldots , r)\}$, which corresponds with an open set of an irreducible component of $\hh_{r}$. Suppose that there is  another $h$-vector $\{1,a_{1},\ldots ,a_{m}\}$, by \ref{lemidealACM} $s$ is equal to $r$ and $h(r)=0$, then by the condition 2, we have that $a_{i}=0$ for all $i\in{\{r,r+1,\ldots ,m\}}$ and therefore  $\{1,a_{1},\ldots ,a_{r-1}\}=\{(1, 2, \ldots , r)\}$ and this is the only possible $h$-vector in $\hh_{r}$ of ACM curves. Then $\cc_{r}$ is the only irreducible component in $\hh_{r}$ of ACM curves.\\
\end{proof}

\begin{rmk} \label{incidence}
For $r >0$, we can consider the incidence variety  $W_{r}$ defined by:

$$\ \xymatrix{&&&W_{r}:=\{ (C_{r},C_{r+1})\in{\mathscr{C}_{r}\times \mathscr{C}_{r+1}}|C_{r}\cup C_{r+1}=X\cap X^{\prime}\} \ar[ld]_{\pi_{0}}\ar[rd]^{\pi_{1}}&&&\\
&&\cc_{r} & &\cc_{r+1}&&}$$
If $C_{r}\in{\cc_{r}}$, then $C_{r+1}$ denotes a curve linked to $C_{r}$ by the complete intersection of two surfaces $X$ and $X^{\prime}$ of degree $r+1$; that means, $C_{r+1}\in{\mathscr{L}_{r+1}\{C_{r}\}}$. A consequence of Theorem \ref{acm} is that $\cc_{r+1}=\mathscr{L}_{r+1}\cc_{r}$ and therefore we compute the dimension and verify that it has the expected dimension.

\begin{lem}For every $r\geq 1$ we have that
$$dim\  \cc_{r}=4d_{r}=2r(r+1).$$
\end{lem}
\begin{proof}
    This can be computed directly from \cite{ell}*{Thm 2}.
\end{proof}
\end{rmk}

If $r\geq 3$ the Hilbert scheme $\hh_{r}$ is reducible; furthermore, it has a generically nonreduced component of dimension $\frac{3}{2}d_{r}(d_{r}-3)+9-2g_{r}$, \cite{mdp1}*{Thm.4.3}. The existence of more than one component complicates verifying that  $\mathscr{L}\mathscr{B}\subseteq \overline{\cc_{r}}$ for a family $ \mathscr{B}$ in the component $\overline{\cc_{r-1}}$. An important part of this work is to verify these contentions for some special families.\\

Now we investigate the relation between the ideal sheaves of linked curves in order to compute the dimension of a linked family $\mathscr{L}\cc$ in terms of the dimension of the family $\cc$. For this, we start with a curve $C_{r}$ in $\mathscr{H}_{r}$, and suppose that there exists a curve $D$ in $\PPP$ and surfaces $X$ and $X^{\prime}$ of degrees $t$ and $s$ respectively such that $C_{r}$ and $D$ are linked by the complete intersection of $X$ and $X^{\prime}$. Assume that $X$ is smooth. Then, we can consider $C_{r}$ as a divisor in $X$ with class $X\cap X^{\prime}-D=sL-D$, where $L$ is a hyperplane section of $X$. 

\begin{lem} Under the above hypotheses, the following equality is satisfied:
\begin{equation} 	  \label{eqa}
\begin{split}
    h^{0}(X,\mathcal{O}_{X}(-C_{r})(m))&-h^{1}(X,\mathcal{O}_{X}(-C_{r})(m))\\
      &=- h^{0}(X,\mathcal{O}_{X}(-D)(s+t-m-4))+\frac{r(r-m-2)(r-m-1)}{2}.
 \end{split}
\end{equation}
\end{lem}  
\begin{proof}
    From de exact sequence of the ideal of $C$ in $X$ and Riemann Roch we obtain the following equation
    \begin{equation} \label{eqRR}
        \chi \ii_{C/X}(m)=\chi \mathcal{O}_{X}(m) -\chi \mathcal{O}_{C}(m)=\binom{m+3}{3}-\binom{m-r+3}{3}-(1-g_{r}+nd_{r})
    \end{equation}
    by Serre duality, we know that $H^{2}(X,\ii_{C/X}(m))\cong H^{0}(X,\mathcal{O}_{X}((s+t-4-m)L-D)=H^{0}(X,\ii_{D/X}(s+t-4-m))$, then, substituting in the equation (\ref{eqRR}) and developing the right side we have the result.\\
\end{proof}

We want to study the cohomology of the ideal sheaf of a curve in $\PPP$, thus let us write the cohomology of the sheaf $ \mathcal{O}_{X}(-C_{r})(m)=\mathcal{O}_{X}(mL- C_{r}))$, over the surface $X$, in terms of the cohomology of the sheaf $\mathscr{I}_{C_{r}}(m)$ over $\PP^{3}$.

From the cohomology of the sequence of the surface $X$ of degree $r$ in $\mathds{P}^{3}$
 \[0\to \mathcal{O}_{\PP^{3}}(m-r) \to \mathcal{O}_{\PP^{3}}(m) \to \mathcal{O}_{X}(m) \to 0,\]
  we have that
 \begin{equation}\label{eq1}
H^{1}(X,\mathcal{O}_{X}(m))\cong H^{1}(\mathds{P}^{3},\mathcal{O}_{\mathds{P}^{3}}(m))=0 \quad \text{for all }m,
 \end{equation}
 and
 
  \begin{align}  
h^{0}(X,\mathcal{O}_{X}(m))&= h^{0}(\mathds{P}^{3},\mathcal{O}_{\mathds{P}^{3}}(m))- h^{0}(\mathds{P}^{3},\mathcal{O}_{\mathds{P}^{3}}(m-r))  &\quad  \quad \text{for all } m\geqslant r-3 \\
&=\binom{m+3}{3}-\binom{m-r+3}{3}& \quad  \quad \text{for all } m\geqslant r-3.  \label{eq2}
 \end{align}

From the exact sequence of a curve $C\subseteq \mathds{P}^{3}$, 
\[0\to \mathscr{I}_{C}(m) \to \mathcal{O}_{\PP^{3}}(m) \to \mathcal{O}_{C}(m) \to 0\]
gives us the next equality:
\begin{equation} \label{eq3}
h^{0}(C,\mathcal{O}_{C}(m))=-h^{0}(\mathds{P}^{3},\mathscr{I}_{C}(m))+h^{0}(\mathds{P}^{3},\mathcal{O}_{\mathds{P}^{3}}(m))+h^{1}(\mathds{P}^{3},\mathscr{I}_{C}(m)) \quad \text{for all } m>0.
\end{equation}
For any curve $C\subseteq X$, we can consider the structure exact sequence of $C$ in $X$:
\[0\to \mathcal{O}_{X}(-C) \to \mathcal{O}_{X} \to \mathcal{O}_{C} \to 0\]
That implies the next equality:
\begin{equation} \label{eq4}
h^{0}(C,\mathcal{O}_{C}(m))=-h^{0}(X,\mathcal{O}_{X}(-C)(m))+h^{0}(X,\mathcal{O}_{X}(m))+h^{1}(X,\mathcal{O}_{X}(-C)(m)) \text{ for all } m>0.
\end{equation}
Thus, we can use  equations (\ref{eq3}) and (\ref{eq4}) and combine them with equation (\ref{eq2}) to obtain:
\begin{equation}  \label{eqb}
\begin{split}
 h^{0}(X,\mathcal{O}_{X}&(-C)(m))-h^{1}(X,\mathcal{O}_{X}(-C)(m))\\
&=-\binom{m-r+3}{3}+h^{0}(\mathds{P}^{3},\mathscr{I}_{C}(m))-h^{1}(\mathds{P}^{3},\mathscr{I}_{C}(m)) \quad \text{for all }  m\geqslant r-3.
 \end{split}
\end{equation}

\begin{lem} \label{lem:eq}
Let $C$ be a curve contained in a surface $X$ of degree $r$ in $\mathds{P}^{3}$, then:
\[ h^{0}(X,\mathcal{O}_{X}(mL-C))=h^{0}(\mathds{P}^{3},\mathscr{I}_{C}(m))-h^{0}(\mathds{P}^{3},\mathcal{O}_{\mathds{P}^{3}}(m-r)). \]
\end{lem}
\begin{proof}
From the last resolutions we have the following inclusions: 
\[u:H^{0}(X,\mathcal{O}_{X}(mL-C))\hookrightarrow H^{0}(X,\mathcal{O}_{X}(mL))\] 
and
\[v:H^{0}(\mathds{P}^{3},\mathscr{I}_{C}(m))\hookrightarrow H^{0}(\mathds{P}^{3},\mathcal{O}_{\mathds{P}^{3}}(mL)).\]
 On the other hand,  we can consider the restriction morphism:
 \[f:H^{0}(\mathds{P}^{3},\mathscr{I}_{C}(m))\to H^{0}(X,\mathcal{O}_{X}(mL-C)).\]
  Furthermore, if $T$ is the equation that defines the surface $X$, then we have a morphism 
  \[g: H^{0}(\mathds{P}^{3},\mathcal{O}_{\mathds{P}^{3}}(m-r)) \to H^{0}(\mathds{P}^{3},\mathscr{I}_{C}(m))\] 
 which maps a section $s$ in $H^{0}(\mathds{P}^{3},\mathcal{O}_{\mathds{P}^{3}}(m-r))$ to the element $T\cdot s$ that vanishes on $C$. These morphisms, and the short exact sequence in cohomology induced by the inclusion of the  surface $X$ in $\mathds{P}^{3}$, imply the next commutative diagram:
  
  \xymatrix{&  &  0\ar[d]& 0\ar[d]& \\ 
& H^{0}(\mathds{P}^{3},\mathcal{O}_{\mathds{P}^{3}}(m-r))\ar@{=}[d]\ar[r]^{\quad g} &  H^{0}(\mathds{P}^{3},\mathscr{I}_{C}(m))\ar[r]^{f \qquad } \ar@{^(->}[d]^{v}& H^{0}(X,\mathcal{O}_{X}(mL-C))\ar@{^(->}[d]^{u} & \\
 0 \ar[r]& H^{0}(\mathds{P}^{3},\mathcal{O}_{\mathds{P}^{3}}(m-r))\ar[r]^{\quad G} &  H^{0}(\mathds{P}^{3},\mathcal{O}_{\mathds{P}^{3}}(m))\ar[r]^{F\quad } & H^{0}(X,\mathcal{O}_{X}(mL)) \ar[r]&0 }
 
A diagram-chasing argument proves Lemma.\\
\end{proof}

Now we compute an equality that allows us to compare the global sections of the ideal sheaf of a curve to those of the ideal sheaf of a curve linked to it. This equality is useful since it allows us to know the minimal degree of a surface that contains a fixed curve if we know the minimal degree of the surfaces that contains a curve directly linked to it.

\begin{prop} \label{star}
Let $C_{r}\in{\mathscr{H}_{r}}$ be a curve and $C_{r-1}\in{\mathscr{H}_{r-1}}$ be a curve linked to $C_{r}$ by the complete intersection of two surfaces $X, X^{\prime}$ of degree $r$. Suppose that $X$ is smooth  and $m\geqslant r-3$ then:
\begin{equation} \tag{$\star$} \label{eqstar}
\begin{split}
\quad h^{0}(\mathds{P}^{3}, \mathscr{I}_{C_{r}}(m))&= h^{1}(\mathds{P}^{3}, \mathscr{I}_{C_{r}}(m))- h^{0}(\mathds{P}^{3}, \mathscr{I}_{C_{r-1}}(2r-m-4))\\
&+ \frac{r(r-m-2)(r-m-1)}{2}+ \binom{m-r+3}{3}  .
 \end{split}
\end{equation}

\end{prop}
\begin{proof}
It follows from (\ref{eqa}), (\ref{eqb}) and Lemma \ref{lem:eq}.\\
\end{proof}

\begin{rmk}
The preliminaries and more details of liaison theory that we present in this section can be found in \cite{mig2}. To the best of our knowledge, Theorem \ref{acm} and Proposition  \ref{star} were not in the literature before this paper.
\end{rmk}

\section{Curves of degree 6 and genus 3} \label{sect6,3}

This Section focuses on the Hilbert scheme $\hh_{3}$ of curves in $\PP^{3}$ of degree $6$ and genus $3$. In \cite{am1}*{Thm.7} is proved that this scheme is reducible and has $3$ components. Also, he gave a description of the modules that occur as Rao modules of locally Cohen-Macaulay curves of degree $6$ and genus $3$.  Now we give a brief description of these three components including the families whose elements have each of the Rao modules that occur in the Hilbert scheme $\hh_{3}$. For this section, we denote by $M(n)$ the graded module $M$ shifted by $n$.
 
 The three components of the Hilbert scheme $\hh_{3}$ are the following:

\begin{enumerate}
    \item \textbf{The main component:} This component is the closure of the ACM curves whose dimension is $24$ and coincides with the closure of the family of smooth curves of degree $6$ and genus $3$. Since will be useful in the next sections, in this component we compute the character of postulation ($\gamma_{c}$) and speciality ($\sigma_{c})$ and the Rao function ($\rho$):
    \begin{center}
\begin{tabular}{ c| c c c}
$n$ &  $\gamma_{c}$ & $\sigma_{c}$ & $\rho$ \\ \hline 
0 &-1 &1 &0\\
1 & -1 & 1& 0  \\
2& -1& 1& 0 \\
3 &3 & -3  & 0
\end{tabular}
\end{center} 
    
    The generic element has the following minimal free resolution:
    \begin{equation*} 
 0\to \mathcal{O}_{\PP^{3}}(-4)^{3}\to \mathcal{O}_{\PP^{3}}(-3)^{4}\to  \mathscr{I}_{C} \to 0. \end{equation*} 
In this component, the generic element has trivial Rao module. Nevertheless, there exist two families of codimension one inside of this component with Rao module of length one:
    \begin{itemize}
        \item  \textbf{Hyperelliptic curves:} Let $\ccc$ be the family of curves of bidegree $(2,4)$ in a smooth quadric surface. The elements in this family are hyperelliptic curves of degree $6$ and genus $3$ and they are not ACM curves. Thus, $\ccc$ is contained in  $\oc-\cc_{3}$.  The dimension of $\ccc$ is $23$, the elements in this family have Rao module isomorphic to $k(-2)$. The character of postulation ($\gamma_{c}$) and speciality ($\sigma_{c})$ and the Rao function ($\rho$) of an element of $\ccc$ are:
          \begin{center}
\begin{tabular}{ c| c c c}
$n$ &  $\gamma_{c}$ & $\sigma_{c}$ & $\rho$ \\ \hline 
0 &-1 &1 &0\\
1 & -1 & 1& 0  \\
2& 0& 1& 1 \\
3 &0 & -3  & 0 \\
4 &3 & 0& 0\\
5 & -1& 0& 0
\end{tabular}
\end{center} 
The ideal of a generic element in it has the following  minimal free resolution:
\[ 0\to \mathcal{O}_{\mathds{P}^{3}}(-6)\to \mathcal{O}_{\mathds{P}^{3}}(-5)^{4}\to \mathcal{O}_{\mathds{P}^{3}}(-4)^{3}\oplus \mathcal{O}_{\mathds{P}^{3}}(-2)\to \mathscr{I}^{h}_{3} \to 0.\]
          \item  \textbf{Reducible curves:} Consider a plane $H$ and two  skew lines  $L_{1},L_{2}$ in $\PP^{3}$ that intersect $H$ in two different points $p_{1}$ and $p_{2}$ respectively. If $Q$  is a plane quartic on $H$ that passes through the points $p_{1}$ and $p_{2}$, then the curve $C=L_{1}\cup L_{2}\cup Q$ (see Figure \ref{antenitas}) has degree $6$ and genus $3$. Thus $C\in{\hh_{3}}$. Let $\mathscr{A}$ be the family of these curves. Note that $\aaa$ is contained in the closure of $\cc_{3}$ because such curves are smoothable by \cite{harhir}*{Cor.4.3}. The dimension of the family $\aaa$ is $23$, the elements in this family have Rao module isomorphic to $k(-1)$. The character of postulation ($\gamma_{c}$) and speciality ($\sigma_{c})$ and the Rao function ($\rho$) of an element of $\aaa$ are:
            \begin{center}
\begin{tabular}{ c| c c c}
$n$ &  $\gamma_{c}$ & $\sigma_{c}$ & $\rho$ \\ \hline 
0 &-1 &1 &0\\
1 & -1 & 2& 1  \\
2& -1& -2& 0 \\
3 &3 & 0  & 0 \\
4 &0 & -1& 0
\end{tabular}
\end{center} 
          The ideal of a generic element in it has minimal free resolution:
\[  0\to \mathcal{O}_{\mathds{P}^{3}}(-5)\to \mathcal{O}_{\mathds{P}^{3}}(-5)\oplus \mathcal{O}_{\PP^{3}}(-4)^{4}\to \mathcal{O}_{\mathds{P}^{3}}(-4)\oplus \mathcal{O}_{\mathds{P}^{3}}(-3)^{4}\to   \mathscr{I}_{C} \to 0. \]
\begin{figure}[h]
\centering
\graphicspath{ {Tesis/IMCT/} } \includegraphics[scale=.3]{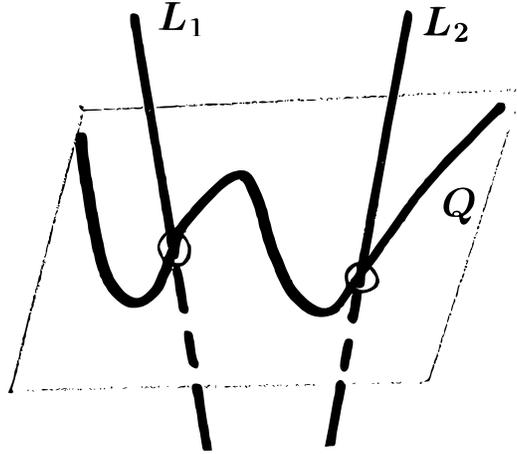}
\caption{$C=L_{1}\cup  L_{2}\cup Q$ is a generic element of $\aaa$}\label{antenitas}
\end{figure}
\end{itemize}

\item \textbf{The component of reducible curves:} We denote this component by $\mathscr{R}_{3}$, the generic element is the union of a plane quartic with a conic that intersects it at a point. The dimension of this component is $24$, the Rao module that appears for the elements of this component is $M=k[x,y,z,w](-1)/\langle x,y,z,w^{3}\rangle$ by \cite{am1}*{Thm.4} and the character of postulation of this family is $\gamma_{C}(0)=\gamma_{C}(1)=\gamma_{C}(6)=-1, \gamma_{C}(1)=1, \gamma_{C}(5)=2$ and $0$ in another case. The minimal free resolution of the ideal of an element in $\mathscr{R}_{3}$ is:
\[ 0\to\mathcal{O}_{\mathds{P}^{3}}(-7)\to\mathcal{O}_{\mathds{P}^{3}}(-6)^{3}\oplus \mathcal{O}_{\PP^{3}}(-4)\to \mathcal{O}_{\mathds{P}^{3}}(-5)^{2}\oplus \mathcal{O}_{\mathds{P}^{3}}(-3)\oplus \mathcal{O}_{\mathds{P}^{3}}(-2)\to \mathscr{I}_{C} \to 0.\]

\item \textbf{The extremal component: } By \cite{mdp1} there exists a component $\mathscr{E}_{3}$ generically nonreduced of dimension $30$, whose generic element is an extremal curve. The Rao module of a generic element in this component is $M=k[x,y,z,w](2)/\langle x,y,F,G\rangle$ with $F,G$ polynomials of degree $3$ and $7$ respectively and the character of postulation of this family is $\gamma_{C}(0)=\gamma_{C}(1)=\gamma_{C}(9)=-1, \gamma_{C}(2)=\gamma_{C}(6)=\gamma_{C}(8)=1$ and $0$ in another case. The minimal free resolution of the ideal of an element in $\mathscr{E}_{3}$ is:
\[ 0\to \mathcal{O}_{\mathds{P}^{3}}(-10)\to \mathcal{O}_{\mathds{P}^{3}}(-9)^{2}\oplus \mathcal{O}_{\PP^{3}}(-7)\to \mathcal{O}_{\mathds{P}^{3}}(-8)\oplus \mathcal{O}_{\mathds{P}^{3}}(-6)\oplus \mathcal{O}_{\mathds{P}^{3}}(-2)^{2}\to  \mathscr{I}_{C} \to 0. \]
\end{enumerate}

For the rest of the section, we focus on the main component $\oc$. Note that in the case when $r$ is $2$ we have that $\hh_{2}=\overline{\cc_{2}}$ and the only lcm curves in this space are the ACM curves. This changes in the case $r=3$, as we saw before, the closure of the ACM curves is only one of three components of $\hh_{3}$. And inside of this component, there are exactly two families of codimension one in $(\oc-\cc_{3})^{lcm}$ as the following
Lemma shows:

\begin{lem} \label{union} The set of locally Cohen-Macaulay curves in $\oc-\cc_{3}$ is equal to the set of locally Cohen-Macaulay curves in the union $\overline{\aaa}\cup \overline{\ccc}$. In other words,
$$ (\oc-\cc_{3})^{lcm}=(\overline{\aaa} \cup \overline{\ccc})^{lcm}.$$
Moreover, taking closure we have that
\begin{equation}\label{eqlcm}
   \overline{(\oc-\cc_{3})^{lcm}}=\overline{\aaa} \cup \overline{\ccc}.
\end{equation}
\end{lem}
\begin{proof}The union $(\overline{\aaa} \cup \overline{\ccc})^{lcm}$ is contained in $(\oc-\cc_{3})^{lcm}$ since the elements in $\oc-\cc_{3}$ are curves with non-trivial Rao module. By \cite{am1}*{Thm.7} any curve in $(\oc)^{lcm}$ with non-trivial Rao module is contained in $(\overline{\aaa})^{lcm} $ or $(\overline{\ccc})^{lcm}$. Then $(\oc- \cc_{3})^{lcm}\subseteq (\overline{\aaa} \cup \overline{\ccc})^{lcm}$.\\
\end{proof}

Let $N^{1}_{\mathds{Z}}(\oc)$ be the quotient of the Cartier divisors on $\oc$  modulo numerical equivalence and $N^{1}(\oc)$ be the tensor product of $N^{1}_{\mathds{Z}}(\oc)$ with the field of the real numbers $\mathds{R}$. We do not know the behavior of $\oc$ outside of the family of locally Cohen-Macaulay curves, thus we can not claim that the dimension of $N^{1}(\oc)$ is finite. But the equality (\ref{eqlcm}) allows us to consider $\mathscr{B}$ the normalization of the scheme $\overline{\aaa}\cup\overline{\ccc} \cup \cc_{3}$ as a partial compactification of $\cc_{3}$. In this case $N^{1}(\mathscr{B})$ has finite dimension.

Furthermore, we claim that the classes of the families $\overline{\ccc}$ and $\overline{\aaa}$ are linearly independent in $N^{1}(\mathscr{B})$. To show this, we begin  by defining a birational map: for a generic element in $\cc_{3}$, its ideal is generated by exactly $4$ linearly independent cubics. These four cubics define a linear subspace of dimension $3$ in the space of cubics $\PP^{19}$. Thus we can define the map:

\[ \xymatrix{h: \mathscr{B} \ar@{-->}[r]& \mathds{G}(3,19)} \]
\[\xymatrix{& & C \ar@{|->}[r] &\mathds{P}(H^{0}(\mathds{P}^{3},\mathscr{I}_{C}(3)))}. \]

The map $h$ is not a morphism: by \cite{am1}*{Lem.9}, the intersection of $\overline{\ccc}$ with $\rr_{3}$ is non-empty and the  general element in this intersection has five linearly independent cubics. Similarly, the intersection of $\overline{\ccc}$ and $\overline{\aaa}$ with $\mathscr{E}_{3}$ is not empty and the ideals of elements in this intersection have $7$  linearly independent cubics. Nevertheless, this occurs outside of $\cc_{3}\cup \ccc \cup \mathscr{A}$ which implies that $h$ is well defined over $\mathscr{B}$ (or $\oc$) in codimension one.

On the other hand, observe that for a generic $C$ in $\cc_{3}$, the four defining cubics determine $C$ scheme-theoretically. This implies that the map $h$ is birational. 

We aim to show that this map contracts the divisors $\overline{\ccc}$ and $\overline{\aaa}$. In order to prove this we need the following two Lemmas:

\begin{lem} \label{lema31}
 The map $h$ contracts $\ccc$ to a family of dimension $9$.
\end{lem}
\begin{proof}
Given a quadric in $\mathds{P}_{[x,y,z,w]}^{3}$ we can obtain an element of  the Grasmannian of $3$-planes in the space $\PP^{19}$ of cubics in $\mathds{P}_{[x,y,z,w]}^{3}$ by taking the subspace generated by the quadric multiplied by the four variables $x,y,z,w$. Thus we can define a morphism:
\begin{align*} f&:\mathds{P}(H^{0}(\mathds{P}^{3},\mathcal{O}_{\mathds{P}^{3}}(2))) \to \mathds{G}(3,19)\\
 &\quad q \longmapsto V_{q}:=\langle q\cdot x,q\cdot y,q\cdot z,q\cdot w \rangle.
 \end{align*}
 Set $Q:=Im\,f$. Observe that $f$ is an embedding, thus $Q$ is a subvariety of $\mathds{G}(3,19)$ of dimension $9$. 

The cubics that contain a curve $C$ of $\ccc$ are multiples of the smooth quadric $q_{0}$ that contains $C$. In particular, the elements $q_0\cdot x, q_0\cdot y, q_0\cdot z$ and $q_0\cdot w$ generate the space $H^{0}(\PPP,\mathscr{I}_{C}(3))$. Thus, by the definition of $h$, we have that $h(\ccc)\subseteq Q$. Let us argue that in fact $h$ is dominant onto $Q$. Let $q$ be a smooth quadric in $\mathds{P}(H^{0}(\mathds{P}^{3},\mathcal{O}_{\mathds{P}^{3}}(2)))$. Let $L$ and $L^{\prime}$ be two lines in the same ruling of $q$ and $p$ an irreducible quartic that contains them. The residual curve to $L\cup L^{\prime}$ in the intersection of $q$ and $p$ is a curve of bidegree $(2,4)$ in $q$ and therefore an element of $\ccc$, thus $f(q)\in{h(\ccc)}$. That implies that $\overline{h(\ccc)}=Q$. \\
 
\end{proof}

\begin{lem}\label{lema32}
 The map $h$ contracts $\aaa$ to a family of dimension $11$.
\end{lem}
\begin{proof}
To prove this we give an auxiliary set-theoretic map that helps us understand the image under $h$ of the family $\aaa$: given two lines in general position $L$ and $L^{\prime}$ in $\PP^{3}$, there exist exactly four linearly independent quadric surfaces that contain them. Thus for any linear equation $F$ of $\PP^{3}$ we can multiply $F$ by these quadrics and obtain a $3$-space in the space of cubics in $\PP^{3}$. Therefore, we have the following map:
\begin{align*} g&:H^{0}(\PP^{3},\mathcal{O}_{\PP^{3}}(1))\times \mathds{G}(1,3)^{2}\dashrightarrow \mathds{G}(3,19)\\
 &\quad (F, L, L^{\prime}) \longmapsto \PP( F\cdot H^{0}(\PP^{3},\mathscr{I}_{L\cup L^{\prime}}(2)))
 \end{align*}
 where $F\cdot H^{0}(\PP^{3},\mathscr{I}_{L\cup L^{\prime}}(2))$ denotes the subspace generated by the four quadrics that contain the lines $L$ and $L^{\prime}$ multiplied by the linear equation $F$. The map $g$ is generically injective, then the subvariety $U:=\overline{Im\ g}$ has dimension $11$.\\
 Since a generic curve $C$ in $\aaa$ is the union of a plane quartic $q_{0}$ and two skew lines, each of them intersecting $q_{0}$ at a point, then the cubics in $H^{0}(\PP^{3}, \mathscr{I}_{C}(3))$ must  contain the plane $F$ that contains $q_{0}$. Consequently, any of such cubics have to be $F\cdot Q$ for some $Q$ of degree $2$ that contains the two lines of $C$. This implies $h(\aaa)\subseteq U$. Additionally, given a generic element $g(F,L,L^{\prime})\in{U}$, let $q$ be a general plane quartic in $F$ that passes through the points of intersection of $L$ and $L^{\prime}$ with $F$. Then $q$ union the lines $L$ and $L^{\prime}$ is an element of $\aaa$. That implies $\overline{h(\aaa)}=U$. \\
 \end{proof}

Now we are able to prove the main result of this section:

\begin{thm} \label{li}
The classes of $\overline{\aaa}$ and $\overline{\ccc}$ in $N^{1}(\mathscr{B})$ are linearly independent and each of them spans an extremal ray of the effective cone
\[  \text{Eff}(\mathscr{B})\subseteq N^1(\mathscr{B}).\]
\end{thm}
\begin{proof}
A birational map $\varphi:X -\to Y$ between normal projective varieties is called \textit{contracting} if the inverse map $\varphi^{-1}$ does not contract any divisors. Let  $exc(\varphi)$ be the subcone of $\textit{Eff}(X)$ spanned by the $\varphi$-exceptional effective divisors. By \cite{oka}*{Lem.2.7} the extremal rays of this cone are generated by the $\varphi$-exceptional prime divisors. By \cite{yuke}*{Lem.1.6} $exc(\varphi)$ is an extremal face of the effective cone and therefore the $\varphi$-exceptional prime divisors are extremal rays of $\textit{Eff}(X)$. On the other hand, since two different prime divisors have different supports, the second part of \cite{oka}*{Lem.2.7} proves that any two $\varphi$-exceptional prime divisors are linearly independent. In our case, lemmas \ref{lema31} and \ref{lema32} prove that the birational map $h$ is a divisorial contraction that contracts the divisors $\overline{\aaa}$ and $\overline{\ccc}$ onto the subscheme $\overline{Im\,f}\cup \overline{Im\,g}$. This implies that the divisors  $\overline{\aaa}$ and $\overline{\ccc}$ are $h$-exceptional and the Theorem follows.\\
\end{proof}

Observe that for each $r\geq 3$ we can define the map
\[ \xymatrix{h_{r}: \overline{\cc_{r}} \ar@{-->}[r]& \mathds{G}(r,\binom{r+3}{3})} \]
\[\xymatrix{& & C \ar@{|->}[r] &\mathds{P}(H^{0}(\mathds{P}^{3},\mathscr{I}_{C}(r)))}. \]
Nevertheless, we cannot repeat the argument of Theorem \ref{li}. In general, we do not know if the map $h_{r}$ is well defined in codimension one and which divisors are contracted by $h_{r}$.
\\

The definition of $\mathscr{B}$ and Theorem \ref{li} tell us that $N^{1}(\mathscr{B})$ is generated by the classes of $\overline{\aaa}$, $\overline{\ccc}$ and generators of $N^{1}(\cc_{3})$.
In order to compute the dimension of the space $N^{1}(\cc_{3})$, we consider the incidence variety  $W_{r}$ as in remark \ref{incidence}. For a generic element $C_{r+1}$ in $\cc_{r+1}$, we can identify any element $(C_{r},C_{r+1})$ of the fiber $\pi_{1}^{-1}(C_{r+1})$ with the $2$-plane of surfaces of degree $r+1$ in $H^{0}(\PP^{3},\mathscr{I}_{C_{r+1}}(r+1))$ that determines the union $C_{r}\cup C_{r+1}$. Since the curves in $\cc_{r+1}$ are integral and are not contained in surfaces of degree $r$, these surfaces do not have common components. Thus we have that $$\pi_{1}^{-1}(C_{r+1})\cong G(2,H^{0}(\PP^{3},\mathscr{I}_{C_{r+1}}(r+1))),$$ hence $dim_{\mathds{R}}N^{1}(\cc_{r+1})+1=dim_{\mathds{R}}N^{1}(W_{r})$, \cite{3264}*{pp. 346}.

On the other hand, for a generic element $C_{r}$ of $\cc_{r}$ we can identify any element $(C_{r},C_{r+1})$ of the fiber $\pi_{0}^{-1}(C_{r})$ with the $2$-plane of surfaces of degree $r+1$ in $H^{0}(\PP^{3},\mathscr{I}_{C_{r}}(r+1))$ that determine the union $C_{r}\cup C_{r+1}$. However, there exist elements in $G(2,H^{0}(\PP^{3},\mathscr{I}_{C_{r}}(r+1)))$ whose surfaces have common components. Thus we can identify the fiber with a open set of $G(2,H^{0}(\PP^{3},\mathscr{I}_{C_{r}}(r+1)))$. Therefore, $W_{r}$ is an open set of an incidence correspondence  $\xymatrix{\widetilde{\mathit{W}}_{r}\ar[r]^{\pi}& \cc_{r}}$ that has as fiber the Grasmannian $G(2,H^{0}(\PP^{3},\mathscr{I}_{C_{r}}(r+1)))$ and such that  $\pi|_{W_{r}}=\pi_{0}$. Then $dim_{\mathds{R}}N^{1} (W_{r})\leq dim_{\mathds{R}}N^{1}(\widetilde{\mathit{W}}_{r})= dim_{\mathds{R}}N^{1}(\cc_{r})+1$. 

These inequalities allow us to write the following:
\begin{equation}\label{picnumb}
dim_{\mathds{R}}N^{1}(\cc_{r+1})\leq dim_{\mathds{R}}N^{1}(\cc_{r}).
\end{equation}
\noindent 
\begin{prop} \label{proppic} For all $r$, we have that $dim_{\mathds{R}}\ N^{1}(\cc_{r})=1$.
\end{prop}
\begin{proof} For each $r$ we consider the divisor $H$ that parameterizes the curves in $\cc_{r}$ that intersect a fixed line, the numerical class of this divisor is not trivial, thus $0< dim_{\mathds{R}}N^{1}(\cc_{r})$. On the other hand, we know that $\cc_{1}=\overline{\cc_{1}}=\mathds{G}(1,3)$ and using the inequality (\ref{picnumb}) we have that for all $r$ 
\[0<dim_{\mathds{R}}N^{1}(\cc_{r})\leq dim_{\mathds{R}}N^{1}(\cc_{1})=dim_{\mathds{R}}N^{1}(\mathds{G}(1,\PPP))=1.\]
\end{proof}

In particular, we have that the spaces $\cc_{r}$ are not affine spaces. With the previous Proposition, we can prove the following Corollary.
\begin{cor} \label{T1} The dimension of the vector space $N^{1}(\mathscr{B})$ is $3$.
\end{cor}
\begin{proof}
By Proposition \ref{proppic} and Lemma \ref{union} we have that $dim_{\mathds{R}}N^{1}(\mathscr{B})\leq 3$. Therefore, it is enough to find three linearly independent divisors in $N^{1}(\mathscr{B})$.
Consider $H$ the locus of all curves in $\mathscr{B}$ that intersect a fixed line. This is a divisor of $\mathscr{B}$; we denote its class in $N^{1}(\mathscr{B})$ by $[H]$. Let $[\overline{\ccc}]$ and $[\overline{\aaa}]$ be the classes of the divisors $\overline{\ccc}$ and $\overline{\aaa}$ in $N^{1}(\mathscr{B})$, respectively. We aim to show that $[H], [\overline{\ccc}]$ and $[\overline{\aaa}]$ are linearly independent.
\item Since $[\overline{\ccc}]$ and $[\overline{\aaa}]$ are linearly independent in $N^{1}(\mathscr{B})$ by Theorem \ref{li}, thus  there exist two classes $\alpha$ and $\beta$ in $N_{1}(\mathscr{B})$ such that $[\overline{\aaa}] \cdot \alpha \not=0$, $[\overline{\ccc}]\cdot \alpha=0$, $[\overline{\aaa}]\cdot \beta=0$ and $[\overline{\ccc}]\cdot \beta \not= 0$.
\item On the other hand, we know that the curves in $\cc_{3}$ are flexible curves, i.e., for every twelve general points in $\PP^{3}$ there exists a curve in $\cc_{3}$ passing through them \cite{perr}*{Prop. 5.6, Cor. 5.7, Prop. 5.11.bis}. We use this property to construct a curve in $\cc_{3}$.  Let $Z$ be a set of $11$ points in general position on $\PP^{3}$ and a general line $L$ such that $Z\cap L=\emptyset$. For every $t\in{L}$ we consider a curve $\delta_{t}$ on  $\cc_{3}$ that contains the set of points $Z\cup \{t\}$. This construction gives us a curve $\delta$ in $\cc_{3}$ which does not intersect $\overline{\aaa}\cup \overline{\ccc}$. And, the intersection of $\delta$ with $[H]$ is positive. Let $\gamma$ be the class of $\delta$ in the space of classes of curves $N_{1}(\mathscr{B})$.
\item The following table summarizes the intersection numbers of curves and divisors we have discussed:
\begin{center}
\begin{tabular}{ c| c c c}
$\cdot$ &  $[\overline{\aaa}]$ & $[\overline{\ccc}]$ & $[H]$ \\ \hline 
$\alpha$ &$\not= 0$ &0 &?\\
$\beta$ & 0 & $\not= 0$& ?  \\
$\gamma$& 0& 0& $>0$ \\
\end{tabular}
\end{center} 
Therefore, the classes $[\overline{\aaa}]$, $[\overline{\ccc}]$ and $[H]$ are linearly independents in $N^{1}(\mathscr{B})$ and hence form a base of this space.\\
\end{proof}

Initially, we aim to compute the dimension of $N^{1}(\oc)$ but we do not know how big $\oc - \mathscr{B}$ is or even if $\oc$ is normal. We know that there exists a component G of the Hilbert scheme $Hilb_{p(t),3}$ of dimension 27 that intersects $\oc$, but we can not prove that this is the only component that intersects it. In particular, we do not know how singular $\oc$ is outside of $\mathscr{B}$.

\section{Linked families of codimension one}
In this section, we define families linked to the families $\ccc$ and $\aaa$ and compute their dimensions. Moreover, we verify that the Hilbert scheme $\hh_{4}$ has at least three components. 

As before we would like to study the divisors on the ACM component $\overline{\cc_{r}}$. In the previous section, we studied a divisor in $\oc$ of curves in the linkage class of the disjoint union of two lines, whose Rao module has a length one. We expect that there exists a reducible divisor on $\overline{\cc_{r}}$ that parameterizes curves with Rao module of length one. Since the families $\ccc$ and $\aaa$ are in the linkage class of the disjoint union of two lines, we consider the corresponding linked families in order to exhibit a divisor in $\overline{\cc_{r}}$ whose elements are in the same linkage class. The obstacle for proving that such linked families are divisors in $\overline{\cc_{r}}$ is verifying that these families are contained in the component $\overline{\cc_{r}}$. Let us fix the notation for these linked families in order to investigate some of their properties.

\begin{defi} \label{deffam} Let  $ \mathscr{L}\ccc$ (resp., $ \mathscr{L}\aaa$) be the family of curves in $\hh_{4}$ that are linked to the elements of $\ccc$ (resp., $\aaa$) by the complete intersection of two surfaces of degree $4$. We define recursively for all $r\geq 4$, the family $\mathscr{L}^{r-3}\ccc$ (resp., $ \mathscr{L}^{r-3}\aaa$) in $\hh_{r}$ as the curves linked to curves in the family $\mathscr{L}^{r-4}\ccc$ (resp., $ \mathscr{L}^{r-4}\aaa$) by the complete intersection of two surfaces of degree $r$.
 \end{defi}
 \medskip 
 
 \medskip\noindent
Now we can obtain information of the linked families $\mathscr{L}^{r-3}\ccc$ and $\mathscr{L}^{r-3}\aaa$ from the families $\ccc$ and $\aaa$.
\begin{lem} \label{lemacoh} Let $C$ be a curve  in $\mathscr{L}^{r-3}\ccc$  then:
\begin{itemize}
     \item If $r$ is odd, we have that $H^{1}( \mathscr{I}_{C}(r-1))\cong k$ and $H^{1}( \mathscr{I}_{C}(n))=0 \quad \text{for all } n\not= r-1$.
      \item If $r$ is even, we have that $H^{1}( \mathscr{I}_{C}(r-2))\cong k$ and $H^{1}( \mathscr{I}_{C}(n))=0 \quad \text{for all } n\not= r-2$.
 \end{itemize}
\end{lem}
\begin{proof}  By definition, the curve $C$ is linked to a curve  $C^{\prime}\in{\mathscr{L}^{(r+1)-3}\ccc}$ by the complete intersection of two surfaces of degree $r+1$. By \cite{mdp2}*{III,Prop. 1.2} we have that:
 \begin{equation} \label{eqlinkcurv}
 H^{1}( \mathscr{I}_{C^{\prime}}(n))\cong H^{1}( \mathscr{I}_{C}(2r-2-n))^{\vee}.
 \end{equation}  
We know the Rao module of an element in $\mathscr{C}^{h}$. Thus  inductively with Equation  \ref{eqlinkcurv} we have the result.\\
 \end{proof}

Similar arguments to those of the proof of Lemma \ref{lemacoh} can be used for the families $\mathscr{L}^{r-3}\aaa$.

\begin{lem} \label{lem:cohant} Let $C$ be a curve  in $\mathscr{L}^{r-3}\aaa$ then:
\begin{itemize}
\item If $r$ is odd,  $H^{1}( \mathscr{I}_{C}(r-2))\cong k$ and $H^{1}( \mathscr{I}_{C}(n))=0 \quad \text{for all } n\not= r-2$.
     \item If $r$ is even,  $H^{1}( \mathscr{I}_{C}(r-1))\cong k$ and $H^{1}( \mathscr{I}_{C}(n))=0 \quad \text{for all } n\not= r-1$.
      
 \end{itemize}
\end{lem}
\begin{proof}
The proof is similar to  Lemma \ref{lemacoh}.\\
\end{proof}

Using Proposition \ref{star} we are able to compute the global sections of the sheaf $\mathscr{I}_{C}(m)$ for all positive $m$ and all $r\geq 3$ for a generic element $C$ of $\mathscr{L}^{r-3}\ccc$. This helps us to compute the dimension of the linked families.

\begin{cor} \label{sg} Let $C$ be a generic curve in $\mathscr{L}^{r-3}\ccc$, then we have that:
\begin{itemize}
\item[a)] $ h^{0}(\PP^{3}, \mathscr{I}_{C}(r-3))=0$
\item[b)] $
    h^{0}(\PP^{3},\mathscr{I}_{C}(r-1)) =\left\{
	       \begin{array}{ll}
		 \text{0}     & \mathrm{if} \; r \; \mathrm{is \; even } \\
		\text{1}     & \mathrm{if} \; r \; \mathrm{is \; odd }  
		 \end{array}
	     \right.
   $
\item[c)] $ h^{0}(\PP^{3}, \mathscr{I}_{C}(r))=r+1$
\item[d)]$ h^{0}(\PP^{3}, \mathscr{I}_{C}(r+1))=3r+4$; in general we have:

\[ h^{0}(\PP^{3}, \mathscr{I}_{C}(r+a))=\frac{(a+1)(a+2)}{2}r + \frac{(a+1)(a+2)(a+3)}{6} \quad \text{for all }a\geqslant 1. \]
\end{itemize}

\end{cor}
\begin{proof} 
\begin{itemize}
\item[$a)$]We use induction over $r$: for the case $r=3$ we know that $h^{0}(\PP^{3}, \mathscr{I}_{C_{3}}(0))=0$. Suppose that it is true for all  $s \leqslant r-1$. Let $C_{r}\in{\mathscr{L}^{r-3}\ccc}$. By definition there exists a curve $C_{r-1}\in{\mathscr{L}^{(r-1)-3}\ccc}$ linked to $C_{r}$ by the complete intersection of two surfaces of degree $r$. Applying the formula (\ref{eqstar}) we have that:
\[ h^{0}(\mathds{P}^{3}, \mathscr{I}_{C_{r}}(r-3))= h^{1}(\mathds{P}^{3}, \mathscr{I}_{C_{r}}(r-3))- h^{0}(\mathds{P}^{3}, \mathscr{I}_{C_{r-1}}(r-1))+r.\]
Now, we know that $h^{1}(\mathds{P}^{3}, \mathscr{I}_{C_{r}}(r-3))=h^{1}(\mathds{P}^{3}, \mathscr{I}_{C_{r-1}}(r-1))=0$ and we can use the Proposition \ref{star} with $C_{r-1}$ and other curve $C_{r-2}\in{\mathscr{L}^{(r-2)-3}\ccc}$ linked to $C_{r-1}$  to obtain:
{\small \[ h^{0}(\mathds{P}^{3}, \mathscr{I}_{C_{r}}(r-3))= -(h^{1}(\mathds{P}^{3}, \mathscr{I}_{C_{r-1}}(r-1))- h^{0}(\mathds{P}^{3}, \mathscr{I}_{C_{r-2}}(r-1))+r)+r=h^{0}(\mathds{P}^{3}, \mathscr{I}_{C_{r-2}}(r-1)).\]} The last term is zero by induction.

\item[$c),d)$]  Follows from $a)$ and  Proposition \ref{star}.

\item[$b)$] Let $C_{r}\in{\mathscr{L}^{r-3}\ccc}$. We consider a curve $C_{r-1}\in{\mathscr{L}^{(r-1)-3}\ccc}$ linked to $C_{r}$ as before. From Proposition \ref{star} we have that:
\[ h^{0}(\mathds{P}^{3}, \mathscr{I}_{C_{r}}(r-1))= h^{1}(\mathds{P}^{3}, \mathscr{I}_{C_{r}}(r-1))- h^{0}(\mathds{P}^{3}, \mathscr{I}_{C_{r-1}}(r-3))+0 .\]
  From $a)$, we have that $h^{0}(\mathds{P}^{3}, \mathscr{I}_{C_{r-1}}(r-3))=0$ for all $r$ and  $h^{1}(\mathds{P}^{3}, \mathscr{I}_{C_{r}}(r-1))$ is one if $r$ is odd and zero if $r$ is even by Lemma \ref{lemacoh}.

\end{itemize}
\end{proof}

Now we can prove that the families $ \mathscr{L}^{r-3}\ccc$ are irreducible and compute their dimension.
\begin{prop} \label{dimch}
Let $r\geq 3$, then the family $\mathscr{L}^{r-3}\ccc$ is irreducible of dimension  $2r(r+1)-1$.
\end{prop}
\begin{proof} We use induction over $r$. For the case $r=3$, we know that $\ccc$  is irreducible of dimension $23$. Suppose that $\mathscr{L}^{r-3}\ccc$is an irreducible family of dimension $2r(r+1)-1$. \\
For every $r\geq 3$, we consider the incidence variety 
\[W^{h}_{r}:=\{ (C_{r},C_{r+1})\in{\mathscr{L}^{r-3}\ccc\times \mathscr{L}^{(r+1)-3}\ccc}|C_{r}\cup C_{r+1}=X\cap X^{\prime}\}\] with $X$ and $X^{\prime}$ surfaces of degree $r+1$. That is, the elements of   $W_{r}^{h}$ are pairs of curves in  $\mathscr{L}^{r-3}\ccc\times \mathscr{L}^{(r+1)-3}\ccc$ that are linked by the complete intersection of two surfaces of degree $r+1$. This subvariety comes with  projection morphisms to each factor: 
\[\pi_{0}^{r}:W_{r}^{h} \to \mathscr{L}^{r-3}_{r}\ccc \qquad \text{ and }\qquad \pi_{1}^{r}:W_{r}^{h} \to \mathscr{L}^{(r+1)-3}\ccc. \] 
Using these projection morphisms, we have that:
\begin{align*}
dim \  \mathscr{L}^{(r+1)-3}\ccc&=dim \ W_{r}^{h}- dim \ (\pi_{1}^{r})^{-1} \\
&  =dim \ \mathscr{L}^{r-3}\ccc +dim \  (\pi_{0}^{r})^{-1}-dim \ (\pi_{1}^{r})^{-1}.
\end{align*} 
For a generic element $C_{r}$ in $\mathscr{L}^{r-3}\ccc$ the fibre of $\pi_{0}^{r}$ is the Grassmannian of $2$-planes in the vector space $h^{0}(\PP^{3},\mathscr{I}_{C_{r}}^{h}(r+1))$. Corollary \ref{sg} implies: 
\[dim \  (\pi_{0}^{r})^{-1}=dim\  G(2, h^{0}(\mathscr{I}_{r}^{h}(r+1)))=dim \ G(2, 3r+4)=6r+6. \]
The fibre of $\pi_{1}^{r}$ for a generic element $C_{r+1}$ in $\mathscr{L}^{(r+1)-3}\ccc$ is the Grassmannian of 2-planes in the vector space $h^{0}(\PP^{3},\mathscr{I}_{C_{r+1}}^{h}(r+1))$ and again Corollary \ref{sg} implies:
\[ dim \  (\pi_{1}^{r})^{-1}= dim  \ G(2, h^{0}(\mathscr{I}_{r+1}^{h}(r+1))) =dim \ G(2, r+2)=2r+2 .\]
Thus, we obtain:
\begin{align*}
dim \ \mathscr{L}^{(r+1)-3}\ccc&=dim \ \mathscr{L}^{r-3}\ccc+dim \  (\pi_{0}^{r})^{-1}-dim \ (\pi_{1}^{r})^{-1}\\
&= 2r(r+1)-1+(6r+6)-(2r+2)\\
&=2r^{2}+6r+3\\
&=2(r+1)(r+2)-1.
\end{align*}
On the other hand, the family $\ccc$ is irreducible and the Grassmannian is irreducible. Therefore, the fibres of $\pi_{0}^{3}$ and $\pi_{1}^{3}$ are irreducible. In particular that implies that $\mathscr{L}\ccc$ is irreducible and inductively is follows that all families $\mathscr{L}^{r-3}\ccc$ are irreducible.

\end{proof}

\medskip
We can reproduce the proof of Corollary \ref{sg} and Proposition \ref{dimch}  in the case of the families $\mathscr{L}^{r-3}\aaa$ and obtain the following:

\begin{cor}\label{corant}
$ $
\begin{enumerate}
\item Let $C$ be a generic curve in $\mathscr{L}^{r-3}\aaa$, then:
\begin{itemize}
\item[a)] $ h^{0}(\PP^{3}, \mathscr{I}_{C}(r-3))=0$
\item[b)]$h^{0}(\PP^{3},\mathscr{I}_{C}(r-1)) =\left\{
	       \begin{array}{ll}
		 \text{0}     & \mathrm{if}  \; r \;\mathrm{ is \; odd } \\
		\text{1}     & \mathrm{if}  \; r \; \mathrm{ is \;even }  
		 \end{array}
	     \right.$
\item[c)] $ h^{0}(\PP^{3}, \mathscr{I}_{C}(r))=r+1$
\item[d)]$ h^{0}(\PP^{3}, \mathscr{I}_{C}(r+1))=3r+4$; in general we have:

\[ h^{0}(\PP^{3}, \mathscr{I}_{C}(r+a))=\frac{(a+1)(a+2)}{2}r + \frac{(a+1)(a+2)(a+3)}{6} \quad \text{for all }a\geqslant 1. \]
\end{itemize}

  \item The family $\mathscr{L}^{r-3}\aaa$ is irreducible of dimension  $2r(r+1)-1$.

\end{enumerate}
\end{cor}

To finish this section we will verify that indeed $\hh_{4}$ has at least three components:

\begin{prop} \label{componentsC4}
The Hilbert scheme $\hh_{4}$ of curves of degree $10$ and genus $11$ is reducible and has at least the following three components:
\begin{enumerate}
 \item The component of ACM curves, denoted by $\overline{\cc_{4}}$, has dimension $40$.
 \item The component of extremal curves, denoted by $\mathscr{E}_{4}$, has dimension $92$.
 \item A component $\mathscr{R}$ of dimension at least $46$ that contains a family that parametrizes the union of two disjoint plane quintic curves.
\end{enumerate}
\end{prop}
\begin{proof}
The family of ACM curves is a component of dimension $40$ by Theorem \ref{acm}. The existence of the component $\mathscr{E}_{4}$ is a result from \cite{mdp1}*{Thm.4.3}.

Set $Q$ the family that parametrizes the union of two disjoint plane quintic curves. This family is irreducible of dimension $46$.

The family $Q$ is not contained in $\overline{\cc_{4}}$ since the dimension of $\overline{\cc_{4}}$ is $40$. On the other hand, the Rao module of an element of $Q$ has length $25$ and by \cite{mdp4} the Rao module of an extremal curve has length $425$. Thus, by semi-continuity, $Q$ cannot be contained in the component $\mathscr{E}_{4}$.

Therefore, there exists a component $\mathscr{R}$ different from $\overline{\cc_{4}}$ and $\mathscr{E}_{4}$ that contains $Q$.\\
\end{proof}

\section{Families in the closure of ACM curves}

We expected that the families $\overline{\mathscr{L}^{r-3}\ccc}$ and $\overline{\mathscr{L}^{r-3}\aaa}$ are divisors of the component $\overline{\cc_{r}}$. We have proved that the families $\overline{\mathscr{L}^{r-3}\ccc}$ and $\overline{\mathscr{L}^{r-3}\aaa}$ have dimension $dim \overline{\cc_{r}}-1$. Nevertheless, since the spaces $\hh_{r}$ are reducible for all $r\geq 3$ it is not obvious that these families are contained in the closure of $\cc_{r}$. In this section, we will describe two families in each $\hh_{r}$ that allow us to verify that the families $\mathscr{L}^{r-3}\ccc$ and $\mathscr{L}^{r-3}\aaa$ have the same geometric description when they have different parity. That means, the families $\mathscr{L}^{r-3}\ccc$ and $\mathscr{L}^{r-2}\aaa$ have the same description for all $r$. One of such families has a geometric description and the other a cohomological one. These descriptions allow us to prove that both families are contained in the corresponding component.

Let $\hh_{r}^{smooth}$ be the open set of all smooth curves on $\hh_{r}$.
\begin{itemize}
\item Denote by $\mathscr{D}_{r-1}$ be the closure of the set of all smooth curves in $\hh_{r}$ that lie in a surface of degree $r-1$. That means:
$$\mathscr{D}_{r-1}:=\overline{\{C\in{\hh_{r}^{smooth}}|h^{0}(\PP^{3},\mathscr{I}_{C}(r-1))=1\}}\subseteq \hh_{r}.$$

\item Let $\mathscr{M}_{r}$ be the closure in $\hh_{r}$ of the set of all smooth curves such that the canonical divisor minus $r-2$ times the hyperplane section $H$ has one section and  the rank of their H-R module is $1$. That means:
$$\mathscr{M}_{r}:=\overline{\{C\in{\hh_{r}^{smooth}}|h^{0}(C,K_{C}-(r-2)H)=1\quad \text{ and }\quad lgth( M(C))=1 \}}\subseteq \hh_{r}.$$
\end{itemize}

The relation between these two families and the families $\mathscr{L}^{r-3}\ccc$ and $\mathscr{L}^{r-3}\aaa$ defined in \ref{deffam} is stated in the next Theorem:

\begin{thm}\label{teodes} 
$ $
\begin{enumerate}
\item If $r$ is an odd number, then the closure of the family $\mathscr{L}^{r-3}\ccc$ is equal to the family $\mathscr{D}_{r-1}$ and the closure of the family  $\mathscr{L}^{r-3}\aaa$ is an irreducible component of  the family $\mathscr{M}_{r}$.

\item If $r$ is an even number, then the closure of the family $\mathscr{L}^{r-3}\aaa$ is equal to the family $\mathscr{D}_{r-1}$ and the closure of the family  $\mathscr{L}^{r-3}\ccc$ is an irreducible component of  the family $\mathscr{M}_{r}$.
\end{enumerate}
\end{thm}
\begin{proof} Suppose that $r$ is an odd number.
\item First we proof that $\overline{\mathscr{L}^{r-3}\ccc}=\mathscr{D}_{r-1}$. The set $\mathscr{L}^{r-3}\ccc \cap  \hh_{r}^{smooth}$ is dense in $\mathscr{L}^{r-3}\ccc$ and by Corollary \ref{sg} we have that $\mathscr{L}^{r-3}\ccc \cap  \hh_{r}^{smooth} \subseteq  \mathscr{D}_{r-1}$. To prove the other containment  we use induction over $s$, with  $r=2s+1$. The base case holds by definition. Let us assume it holds for $r-2=2(s-1)+1$.
    
Let $C\in{\hh_{r}^{smooth}\cap \mathscr{D}_{r-1}} $, that means, $h^{0}(\PP^{3},\mathscr{I}_{C}(r-1))=1$. We can then find two surfaces without common components  $X$ and $Y$ of degree $r$ that contains $C$. By Bertini's Theorem we can assume that $X$ is smooth. Then we have a curve $C^{\prime}\in{\mathscr{H}_{r-1}}$ linked to $C$ by the complete intersection of  $X$ and  $Y$ and by the Proposition \ref{star} we have that:
    \begin{align*}
    h^{0}(\PP^{3},\mathscr{I}_{C^{\prime}}^{h}(r-1))&= h^{1}(\PP^{3}\mathscr{I}_{C}(r-3))-h^{0}(\mathscr{I}_{C}(r-3))+r.
   \end{align*}

Since $h^{0}(C,\mathscr{I}_{C}(r-1))=1$, we have that $h^{0}(C,\mathscr{I}_{C}(r-3))=0$, thus $h^{0}(\mathscr{I}_{C^{\prime}}^{h}(r-1))> r-1$. Then we can link $C^{\prime}$ to a curve $C^{\prime \prime}\in{\mathscr{H}_{r-2}}$ by the complete intersection of two surfaces without common components $X^{\prime}$ and $Y^{\prime}$ of degree $r-1$.
    
Again by Proposition \ref{star}, we have that:
\[h^{0}(\PP^{3},\mathscr{I}_{C}(r-1))=h^{1}(\PP^{3},\mathscr{I}_{C}(r-1))-h^{0}(\PP^{3},\mathscr{I}_{C^{\prime}}(r-3))+ h^{0}(\PP^{3},\mathscr{I}_{C^{\prime \prime}}(r-3)).\]

By construction $C$ and $C^{\prime}$ are linked by the complete intersection of surfaces of degree $r$, which implies $H^{1}(\PP^{3},\mathscr{I}_{C^{\prime}}(r-3))\cong H^{1}(\PP^{3},\mathscr{I}_{C}(r-1))^{\vee}$. Therefore $h^{0}(\PP^{3},\mathscr{I}_{C^{\prime \prime}}(r-3))=h^{0}(\PP^{3},\mathscr{I}_{C}(r-1))=1$. Then, by induction $C^{\prime \prime}\in{\mathscr{D}_{(r-2)-1}=\mathscr{L}^{(r-2)-3}\ccc}$ and the definition of our families it follows that $C^{\prime}\in{\mathscr{L}^{(r-1)-3}\ccc}$ and thus $C\in{\mathscr{L}^{r-3}\ccc}$.\\ 
   
   \item Now we prove that $\overline{\mathscr{L}^{r-3}\aaa}$ is a component of $\mathscr{M}_{r}$. Let us start with a generic element $C$ in $\mathscr{L}^{r-3}\aaa$. It follows that $h^{1}(\PP^{3},\mathscr{I}_{C}(r-2))=1$ and $h^{1}(\PP^{3},\mathscr{I}_{C}(n))=0$ for all $n\not= r-2$ by Lemma \ref{lem:cohant}. Hence the H-R module of $C$ has length one. On the other hand, by the exact sequence:
   \[0\to \mathscr{I}_{C}(r-2) \to \mathcal{O}_{\PP^{3}}(r-2) \to \mathcal{O}_{C}(r-2) \to 0,\]
   
   we have the next exact sequence in cohomology:\\
   \begin{equation}\label{sucex}
\xymatrix{0 \ar[r] &H^{0}(\PP^{3},\mathscr{I}_{C}(r-2))\ar[r]& H^{0}(\PP^{3},\mathcal{O}_{\PP^{3}}(r-2)) \ar[r]& H^{0}(C,\mathcal{O}_{C}(r-2)) \ar@(dr,ul)[dll]&\\
 &H^{1}(\PP^{3},\mathscr{I}_{C}(r-2))\ar[r] & H^{1}(\PP^{3},\mathcal{O}_{\PP^{3}}(r-2)) \ar[r] &H^{1}(C,\mathcal{O}_{C}(r-2)) \ar@(dr,ul)[dll] &\\
 &H^{2}(\PP^{3},\mathscr{I}_{C}(r-2))\ar[r]& H^{2}(\PP^{3},\mathcal{O}_{\PP^{3}}(r-2)) \ar[r]&H^{2}(C,\mathcal{O}_{C}(r-2)) \ar[r]& \ldots}   
   \end{equation} 
   
   By Serre duality $H^{1}(C,\mathcal{O}_{C}(r-2)) \cong H^{0}(C,\omega_{C}\otimes \mathcal{O}_{C}(r-2)^{\vee})$, thus using Riemann-Roch we have that:
   \begin{equation}\label{equ1}
  \chi \mathcal{O}_{C}(r-2)+1 =(r-2)d_{r}+1-g_{r}+1=\binom{r+1}{3}+1 
   \end{equation}
   
   On the other hand, $h^{0}(\PP^{3},\mathscr{I}_{C}(r-2))=0$ by Corollary \ref{corant}, thus from the sequence (\ref{sucex}) it follows that: 
   \begin{equation} \label{equ2}
 h^{0}(C,\mathcal{O}_{C}(r-2))= h^{0}(\PP^{3},\mathcal{O}_{\PP^{3}}(r-2))+h^{1}(\PP^{3},\mathscr{I}_{C}(r-2))=\binom{r-2+3}{3} + 1.
  \end{equation}
   
  Then the equations (\ref{equ1}) and (\ref{equ2}) imply that $h^{1}(C,\mathcal{O}_{C}(r-2))=1$; therefore $C\in{\mathscr{M}_{r}}$. 
 
 Since $\mathscr{L}^{r-3}\aaa$ is irreducible, there exists an irreducible component of $\mathscr{M}$ that contains it. We denote this component by  $\mathscr{M}^{\prime}$. Let $C$ be a generic element of $\mathscr{M}^{\prime}$. Observe that we have an exact sequence as (\ref{sucex}) for the ideal of $C$. By Serre duality we have that $h^{1}(C,\mathcal{O}_{C}(r-2))= h^{0}(C,K_{C}-(r-2)H)=1$. Thus 
 \begin{align*}
 h^{0}(C,\mathcal{O}_{C}(r-2))-1=h^{0}(C,\mathcal{O}_{C}(r-2))-h^{1}(C,\mathcal{O}_{C}(r-2))
=\chi \mathcal{O}_{C}(r-2)\\=h^{0}(\PP^{3},\mathcal{O}_{\PP^{3}}(r-2))=h^{0}(C,\mathcal{O}_{C}(r-2))-h^{1}(\PP^{3},\mathscr{I}_{C}(r-2)).
\end{align*}
Therefore $h^{1}(\PP^{3},\mathscr{I}_{C}(r-2))=1$  but $lgthM(C)=1$ implies that $C\in{\mathscr{L}^{r-3}\aaa}$.
 \item The proof in which $r$ is even is analogous.

\end{proof}

With these descriptions we can verify that these families are actually related by an elemental biliaison (cf. \cite{mdp2}*{III def. 2.1}).
\begin{prop} \label{prop1}
    For all $r\geq 3$. Given a curve $C$ in $\mathscr{D}_{r-1}\subseteq \hh_{r}$, then there exists a curve $C^{\prime}$ in $\mathscr{D}_{r}\subseteq \hh_{r+1}$ that is obtained from $C$ by an elementary biliason of degree $r+1$ and height $1$, or, abbreviated an elementary biliason $(r+1,1)$.
\end{prop}
\begin{proof}
    Since $h^{0}(\PPP, \mathscr{I}_{C}(r))=r+1$, we can consider a surface $Q$ of degree $r+1$ that contains $C$ and $S$ a surface of degree $r$ that contains $C$ without common components with $Q$. Then $Q\cap S=C\cup C_{1}$, with $C_{1}\in{\hh_{r}}$. Let $S^{\prime}$ a surface of degree $r+1$ that contains $C_{1}$ without common components with $S$. Thus $S\cap S^{\prime}=C_{1}\cup C^{\prime}$ with $C^{\prime}\in{\hh_{r+1}}$. By definition $C^{\prime}$ is obtained from $C$ by an elementary biliason $(r+1,1)$. Therefore is enough to prove that $C^{\prime}$ is an element of $\mathscr{D}_{r}$.
By \cite{mdp2}*{III 3.4} we have that:
    $$ h^{0}(\PPP,\mathscr{I}_{C^{\prime}}(r)=h^{0}(\PPP,\mathscr{I}_{C}(r-1))+ \binom{r-(r+1)+2}{2}=h^{0}(\PPP,\mathscr{I}_{C}(r-1))$$
   the last term is equal to one since $C$ is an element of $\mathscr{D}_{r-1}$ and by definition this implies that $C^{\prime}\in{\mathscr{D}_{r}}$.\\
\end{proof}
This proposition allows us to compute the character of postulation and speciality of the curves in the families $\mathscr{D}_{r-1}$:

\begin{cor} \label{corresolutionsdr}
    For $r\geq 3$, let $C$ a curve in $\mathscr{D}_{r-1}$ then its character of postulation ($\gamma_{C}$) and speciality ($\sigma_{C}$) are given by:
              \begin{center}
\begin{tabular}{ c| c c }
$n$ &  $\gamma_{C}$ & $\sigma_{C}$ \\ \hline 
0 &-1 &1 \\
1 & -1 & 1  \\
$\vdots $& $\vdots$ & $\vdots$  \\
r-2 &-1 & 1   \\
r-1 &0 & 1\\
r & r-3& -r \\
r+1 &3 &0 \\
r+2 & -1& 0
\end{tabular}
\end{center} 
and the minimal free resolution of $C$ is:
{\scriptsize  \begin{equation} \tag{$\#$} \label{resolDr}
    \xymatrix{ 0\to \mathcal{O}_{\PPP}(-(r+3)) \to \mathcal{O}_{\PPP}(-(r+2))^{4}\oplus \mathcal{O}_{\PPP}(-(r+1))^{r-6}\to \mathcal{O}_{\PPP}(-r)^{r-3}\oplus \mathcal{O}_{\PPP}(-(r-1))\ar[r]^{\qquad \qquad \qquad \qquad \qquad \qquad \qquad \qquad \qquad \pi_{r}}&  \ii_{C} \to 0 }\end{equation}}
\end{cor}
\begin{proof}
    The first part follows recursively from the character of postulation and speciality of a curve in $\ccc$ and Proposition 3.4 and 3.6 from \cite{mdp2}*{III}.

    For the resolution we prove by induction. The base case ($r=3$) is given in the section \ref{sect6,3}. Suppose that $C_{r}\in{\mathscr{D}_{r-1}}$ has the resolution (\ref{resolDr}), therefore the resolution of type $E$ of $\ii_{C_{r}}$ (cf. \cite{mdp2}*{II,Def 3.7}) is given by:
    \begin{equation} \tag{$\triangle_{d}$} \label{resdr}
        0 \to E_{r} \to \mathcal{O}_{\PPP}(-r)^{r-3}\oplus \mathcal{O}_{\PPP}(-(r-1))\to \ii_{C_{r}} \to 0
    \end{equation}
    with $E_{r}=Ker\,\pi_{r}$ given by the exact sequence:
    $$ 0 \to \mathcal{O}_{\PPP}(-(r+3))\to \mathcal{O}_{\PPP}(-(r+2))^{4}\oplus \mathcal{O}_{\PPP}(-(r+1))^{r-6} \to E_{r}\to 0$$
  Let $C_{r+1}\in{\mathscr{D}_{r}}$, by \ref{prop1} this curve is obtained from a curve $C_{r}\in{\mathscr{D}_{r-1}}$ by elementary biliason (r+1,1), then by \cite{mdp2}*{II,Prop 4.3} the resolution of type $E$ of $C_{r+1}$ are given by:
 {\small \begin{equation*}
        0 \to (E_{r}\oplus \mathcal{O}_{\PPP}(-(r+1)))(-1) \to (\mathcal{O}_{\PPP}(-r)^{r-3}\oplus \mathcal{O}_{\PPP}(-(r-1)))(-1)\oplus \mathcal{O}_{\PPP}(-(r+1))\to \ii_{C_{r+1}} \to 0
    \end{equation*}}
    with $E_{r+1}=(E_{r}\oplus \mathcal{O}_{\PPP}(-(r+1)))(-1)$ defined by the exact sequence:
 {\small   $$ 0 \to \mathcal{O}_{\PPP}(-(r+3))(-1)\to (\mathcal{O}_{\PPP}(-(r+2))^{4}\oplus \mathcal{O}_{\PPP}(-(r+1))^{r-6}\oplus \mathcal{O}_{\PPP}(-(r+1)))(-1) \to E_{r+1}\to 0.$$}
    With the last two sequences we conclude that the resolution of $C_{r+1}$ is the expected resolution given by the statement.\\
\end{proof}

\begin{cor}
    For all $r\geq 3$. Any element of $\mathscr{D}_{r-1}$ are linked to an element of $\mathscr{M}^{\prime}_{r}$.
\end{cor}
\begin{proof}
    With the notation of the proof of the Proposition \ref{prop1}. The curve $C$ is linked to the curve $C_{1}$, and $C_{1}$ is linked by the complete intersection of two surfaces of degree $r+1$ to $C^{\prime}$ that is an element of $\mathscr{D}_{r}$ and by \ref{teodes} this family is equal to $\mathscr{L}\mathscr{M}^{\prime}_{r}$, thus $C_{1}\in{\mathscr{M}^{\prime}_{r}}$.\\
\end{proof}

\begin{prop} \label{prop2}
    For all $r\geq 3$. Given a curve $C$ in $\mathscr{M}^{\prime}_{r}\subseteq \hh_{r}$, then there exists a curve $C^{\prime}$ in $\mathscr{M}^{\prime}_{r+1}\subseteq \hh_{r+1}$ that is obtained from $C$ by an elementary biliason $(r+1,1)$.
\end{prop}
\begin{proof}
    We have the same construction of the proof of the Proposition \ref{prop1}. But in this case we use the Proposition \cite{mdp2}*{III 3.5} and obtain:
  \begin{align*}
   h^{0}(C^{\prime},K_{C^{\prime}}-((r+1)-2)H^{\prime})&=h^{2}(\PPP,\mathscr{I}_{C^{\prime}}(r-1))\\
   &=h^{2}(\PPP,\mathscr{I}_{C}((r-1)-1))+ \binom{(r+1)-(r-1)-1}{2}\\
   & \quad- \binom{-(r-1)-1}{2} \\
   &=h^{2}(\PPP,\mathscr{I}_{C}((r-1)-1))=h^{0}(C, K_{C}-(r-2)H)
   \end{align*}
  but $C\in{\mathscr{M}^{\prime}_{r}}$, then $h^{0}(C, K_{C}-(r-2)H)=1$ and since both elements are in the same linked class and the Rao module of $C$ has length one, then the Rao module of $C^{\prime}$ has length one, thus by definition $C^{\prime}$ is an element of $\mathscr{M}^{\prime}_{r+1}$.\\
\end{proof}

\begin{cor}
For $r\geq 3$, let $C$ a curve in $\mathscr{M}_{r}$ then its character of postulation ($\gamma_{C}$) and speciality ($\sigma_{C}$) are given by:
              \begin{center}
\begin{tabular}{ c| c c }
$n$ &  $\gamma_{C}$ & $\sigma_{C}$ \\ \hline 
0 &-1 &1 \\
1 & -1 & 1  \\
$\vdots $& $\vdots$ & $\vdots$  \\
r-3 & -1 & -1 \\
r-2 &-1 & 2   \\
r-1 &-1 & -2\\
r & r& -(r-3) \\
r+1 &0 &-1 
\end{tabular}
\end{center} 
and the minimal free resolution of $C$ is:
{\small  \begin{equation} \tag{$\# \#$} \label{resolmr}
    \xymatrix{ 0\to \mathcal{O}_{\PPP}(-(r+2)) \to \mathcal{O}_{\PPP}(-(r+2))\oplus \mathcal{O}_{\PPP}(-(r+1))^{r}\to \mathcal{O}_{\PPP}(-r)^{r+1}\ar[r]^{\qquad \qquad \qquad \qquad \qquad \qquad \qquad \pi_{r}}&  \ii_{C} \to 0 }\end{equation}}
\end{cor}
\begin{proof}
    The proof is analogous to corollary \ref{corresolutionsdr}, thus we only write the resolutions of type $E$ that correspondent: if $C_{r}\in{\mathscr{M}_{r}}$ has the resolution (\ref{resolmr}), therefore the resolution of type $E$ of $\ii_{C_{r}}$ (cf. \cite{mdp2}*{II,Def 3.7}) is given by:
    \begin{equation} \tag{$\triangle_{m}$} \label{resmr}
        0 \to E_{r} \to \mathcal{O}_{\PPP}(-r)^{r+1}\to \ii_{C_{r}} \to 0
    \end{equation}
    with $E_{r}=Ker\,\pi_{r}$ given by the exact sequence:
    $$ 0 \to \mathcal{O}_{\PPP}(-(r+2))\to \mathcal{O}_{\PPP}(-(r+2))\oplus \mathcal{O}_{\PPP}(-(r+1))^{r} \to E_{r}\to 0$$
  Let $C_{r+1}\in{\mathscr{M}_{r+1}}$, by \ref{prop2} this curve is obtained from a curve $C_{r}\in{\mathscr{M}_{r-1}}$ by elementary biliason (r+1,1), then by \cite{mdp2}*{II,Prop 4.3} the resolution of type $E$ of $C_{r+1}$ are given by:
 {\small \begin{equation*}
        0 \to (E_{r}\oplus \mathcal{O}_{\PPP}(-(r+1)))(-1) \to (\mathcal{O}_{\PPP}(-r)^{r+1})(-1)\oplus \mathcal{O}_{\PPP}(-(r+1))\to \ii_{C_{r+1}} \to 0
    \end{equation*}}
    with $E_{r+1}=(E_{r}\oplus \mathcal{O}_{\PPP}(-(r+1)))(-1)$ defined by the exact sequence:
 {\small   $$ 0 \to \mathcal{O}_{\PPP}(-(r+2))(-1)\to (\mathcal{O}_{\PPP}(-(r+2))\oplus \mathcal{O}_{\PPP}(-(r+1))^{r}\oplus \mathcal{O}_{\PPP}(-(r+1)))(-1) \to E_{r+1}\to 0.$$}
    With the last two sequences we conclude that the resolution of $C_{r+1}$ is the expected resolution given by the statement.\\
\end{proof}

\begin{cor}
   For all $r\geq 3$. Any element of $\mathscr{M}^{\prime}_{r}$ are linked to an element of $\mathscr{D}_{r-1}$. 
\end{cor}

In summary we have that the families $\mathscr{L}^{r-3}\mathscr{A}$ and $\mathscr{L}^{r-3}\ccc$ are directly linked for all $r\geq 3$.

Now we can prove the Theorem B:

\begin{thm}[Theorem B] \label{corP} For all $r\geq 3$, the families  $\mathscr{L}^{r-3}\aaa$ and $\mathscr{L}^{r-3}\ccc$ are contained in the closure  $\overline{\cc_{r}}$.
\end{thm}
\begin{proof}
  By the Theorem \ref{teodes}, is enough to prove that for all $r\geq 3$, the families $\mathscr{D}_{r-1}$ and $\mathscr{M}_{r}$ are contained in $\overline{\cc_{r}}$. First, we can compute the character of postulation ($\gamma_{r}$) and speciality ($\sigma_{r}=-\gamma_{r}$) of a curve on $\cc_{r}$:
 \begin{center}
\begin{tabular}{ c| c c }
$n$ &  $\gamma_{r}$ & $\sigma_{r}$ \\ \hline 
0 &-1 &1 \\
1 & -1 & 1  \\
$\vdots $& $\vdots$ & $\vdots$  \\
r-1 &-1 & 1\\
r & r& -r \\
\end{tabular}
\end{center} 
Since the character of postulation of a curve in $\mathscr{M}_{r}$ is equal to $\gamma_{r}$, we have that the families $\mathscr{M}_{r}$ and $\cc_{r}$ are contained in the Hilbert scheme $H_{\gamma_{r}}$ (cf. \cite{mdp2}*{VI,Prop 3.9}). Moreover, since the curves $C$ in the family $\mathscr{M}_{r}$ is in the liaison class of the disjoint union of two lines, there exists an integer $l$ such that $\rho_{C}(l)=1$ and is 0 in other cases (with $\rho_{C}$ the Rao function of $C$), in fact in this case, $l=r-2$. Following the notation of \cite{mdp2}*{X} we have from the resolutions \ref{resmr} that $b(l+4)=b(r+2)=1\not= 0$, and by \cite{mdp2}*{X,Prop 4.1} the curves in $\mathscr{M}_{r}$ are contained in the Closure of $\cc_{r}$.

On the other hand, the character of speciality of a curve in $\mathscr{D}_{r-1}$ is equal to $\sigma_{r}$, thus we have that the families $\mathscr{D}_{r-1}$ and $\cc_{r}$ are contained in the Hilbert scheme $H_{\sigma_{r}}$. In this case $l=r-1$ and by the resolution of type $E$ \ref{resdr} and \cite{mdp2}*{X,Prop 1.2} we have that $b^{\prime}(l)=b^{\prime}(r-1)=c(r-1)=1\not= 0$ and by \cite{mdp2}*X,Prop 4.3 we conclude.\\
\end{proof}

\end{document}